\documentclass[12pt]{amsart}
\usepackage{amsfonts,amsmath,amsthm,oldgerm,amssymb,amscd}
\usepackage[all,cmtip]{xy}
\newcommand{\ra}{\rightarrow}

\newcommand{\comment}[1]{}

\newtheorem{theorem}{Theorem}
\newtheorem{prop}{Proposition} 

\newtheorem{cor}{Corollary}

\newtheorem{lemma}{Lemma}

\theoremstyle{definition} 

\newtheorem{defn}{Definition}
\newtheorem{example}{Example}

\newtheorem{remark}{Remark}
\newtheorem*{thmwn}{Theorem}

\newcommand{\C}{\mbox{$\mathbb C$}}     

\entrymodifiers={+!!<0pt,\fontdimen22\textfont2>}

\begin{document}
\title{On Cohomology Theory for Topological Groups}

\author{Arati S. Khedekar}
\author{C. S. Rajan}
\address{School of Mathematics, Tata Institute of Fundamental Research, Homi Bhabha Road
Mumbai- 400005, India.}
\date{}
\email{arati@math.tifr.res.in, rajan@math.tifr.res.in}
\begin{abstract}
We construct some new cohomology theories for
topological groups and Lie groups and study some of its basic properties. For example, we 
introduce a cohomology theory based on measurable cochains which are continuous in a 
neighbourhood of identity. We show that if $G$ and $A$ are locally compact and second countable, then the 
second cohomology group based on  locally continuous measurable cochains as above parametrizes
the collection of locally split extensions of $G$ by $A$. 
\end{abstract}

\maketitle

\section{Introduction}
\label{introduction}
\pagenumbering{arabic}\setcounter{page}{1}
The cohomology theory of topological groups has been studied from different
perspectives by W. T. van Est, Mostow, Moore, Wigner and recently Lichtenbaum 
amongst others. W. T. van Est developed a cohomology theory using  continuous
 cochains in analogy 
with the cochain construction of cohomology theory of finite groups. However,
 this definition of 
cohomology groups has a drawback, in  that it gives long exact sequences of 
cohomology groups only 
for those short exact sequences of modules that are topologically split. 

 Based on a theorem of Mackey ({\it cf.} \cite{dixmier}),
  guaranteeing the existence
 of measurable cross sections for locally compact groups, 
Moore developed a 
cohomology theory of topological groups using measurable cochains in place of 
continuous 
cochains.  This cohomology theory works for the category of Polish groups $G$
and $G$-modules $A$ which are again Polish. We recall,
a topological group $G$ is said to be Polish, if its topology is induced by a complete
separable metric on $G$.
This theory satisfies the nice properties expected from a
cohomology theory ({\it cf.} \cite{moore3}) viz., there exists 
long exact sequences of cohomology groups of a Polish group $G$ 
 corresponding to a short exact sequence of 
Polish $G$-modules, the correct interpretation of the first measurable
cohomology as the space of continuous crossed
homomorphism, when $G$ and $A$ are locally compact; an interpretation 
of the second cohomology  $H^2_m(G, A)$ 
 in terms of topological extensions of $G$ by $A$, etc.  
Here $H^*_m(G,A)$ denotes the Moore cohomology group of a topological group $G$ 
and  a topological $G$-module $A$.
It further agrees with the van Est continuous cohomology groups, when $G$ is 
profinite and the coefficient module $A$ is discrete. 

The cohomology theory developed by Moore has had numerous applications (for some recent applications and also for further results on Moore cohomology groups see \cite{austin}).  
The motivation for us  to consider the cohomology theory of topological groups, is to explore the possibility of deploying such theories to the study of the non-abelian reciprocity laws as conjectured by Langlands, just as the continuous cohomology theory of Galois groups has proved to be immensely successful in class field theory. In this context, the analogues of the Galois group like the Weil group $W_k$ attached to a number field $k$ (or the conjectural Langlands group whose finite dimensional representations are supposed to parametrize automorphic representations) are no longer profinite but are locally compact. Indeed such a motivation led the second author to generalize a classical theorem of Tate on the vanishing of the Schur cohomology groups to the context of Weil groups (\cite{rajan}), to  show that $H^2_m(W_k, \C^*)$ vanishes, where  we impose the trivial module structure on  $\C^*$. 

The immediate inspiration for us is the recent work of Lichtenbaum (\cite{lich}), where he outlines deep conjectures explaining the special values of zeta functions of varieties in terms of Weil-\'{e}tale cohomology. Here the cohomology of the generic fibre turns out to be the cohomology of the Weil group. Lichtenbaum studies the cohomology theory of topological groups from an abstract viewpoint based on the work of Grothendieck (\cite{sga4}), where he embeds the category of $G$-modules in a larger abelian category with sufficiently many injectives. The cohomology groups are then the right derived functors of the functor of invariants (we refer to the paper by Flach (\cite{flach}) for more details and applications to the cohomology of Weil groups). Lichtenbaum imposes a Grothendieck topology on the category of $G$-spaces, where the covers have local sections. The required abelian category is the category of sheaves with respect to this Grothendieck topology. 

In this paper, we introduce a new cohomology theory of topological groups. We modify Moore's construction and impose a local regularity condition on the cochains in a neighbourhood of identity (like continuity or smoothness in the context of Lie groups) but assume the cochains to be measurable everywhere. 
The basic observation which makes this possible is the following: given a 
short exact sequence of Lie groups
\[1\rightarrow G'\rightarrow G \rightarrow G'' \rightarrow 1,\]
there is a continuous section from a neighbourhood of identity in $G''$
to $G$. More generally,
using the solution to Hilbert's fifth problem and the observation 
for Lie groups, Mostert ({\it cf.} \cite{mostert}) showed 
that every short exact sequence of finite dimensional 
locally compact groups 
\[1\rightarrow G'\rightarrow G \overset\pi\rightarrow G'' \rightarrow 1\]
is locally split {\it i.e.}, there exists a continuous section from a
neighbourhood of identity in $G''$ to $G$. 

Define 
the group of $0$-cochains $C^0(G,A)$ to be $A$. For $n\geq 1$, define the 
group $C^n_{lcm}(G,A)$ of locally continuous measurable cochains to be the space of all measurable 
functions $f\colon G^n\ra A$ which are continuous in a neighbourhood of the identity in $G^n$. 
The coboundary map is
given by the standard formula. 
 Now we define our locally continuous 
cohomology theory $H^n_{lcm}(G,A)$ as the cohomology of this cochain complex.
These cohomology groups interpolate the continous cohomology and 
the measurable cohomology theory of Moore: there are natural maps
\[H^n_{c}(G,A)\to H^n_{lcm}(G,A)\to H^n_{m}(G,A),\]
where $H^n_{c}(G,A)$ denotes the continuous cohomology groups of $G$ with values in $A$. 

For the category of Lie groups,
 we replace  continuity with the property of being smooth
 around identity and we define the 
locally smooth measurable cohomology 
theory (denoted by $\lbrace H^n_{lsm}(G,A)\rbrace_{n\geq 0}$ ) 
of a Lie group  $G$ that acts smoothly  on $A$. 
Similarly, we can define locally holomorphic 
measurable cohomology theory (denoted as 
 $\lbrace H^n_{lhm}(G,A)\rbrace_{n\geq 0}$) in the holomorphic 
category, based on measurable cochains holomorphic in a
neighbourhood of identity.

We remark out here that although the cohomology theories developed by Moore and Lichtenbaum seem sufficient for many purposes, the richness of applications of cohomology arises from the presence of different cohomology theories that can be compared to each other. The multiplicity of such theories allow the use of cohomological methods in a variety of contexts. In this regard, we expect that the principle of imposing local regularity on the cochains, will allow it's use in more geometric and arithmetical contexts. For example, it is tantalizing to explore the relationship of these theories to the measurable Steinberg $2$-cocycle (\cite{moore2}), which is continuous on a dense open subset (but not at identity!).

These locally regular cohomology theories can be related  
to the underlying category theoretic properties of the group and it's modules. 
For example, suppose there is an 
extension \[1\ra A\ra E\ra G\ra 1\] of $G$ by $A$ given 
by a measurable $2$-cocycle. From the construction 
of this extension (as given by Moore, {\it cf.} \cite[page 30]{moore3}), it seems 
difficult to relate the topology 
of $E$ to that of $G$ and $A$.  
If $G$ and $A$ are locally compact, it is a difficult theorem of Mackey that  $E$ 
is locally compact ({\it cf. }\cite{mackey}). 
Another difficulty arises, when 
we work with a  Lie group  $G$ and a smooth
 $G$-module $A$. It is not clear when an extension of $G$ by $A$ defined by
a measurable $2$-cocycle is a Lie group. Further, there does not 
seem to be any  relationship 
between the Moore cohomology groups and the cohomology 
groups of the associated Lie Algebra and its module.

We now describe some of our results towards establishing the legitimacy of these theories. 
It can be seen that these locally regular cohomology theories
are cohomological, in that there exists long exact sequence of
cohomology groups associated to locally split short exact sequences
of modules. Further, the zeroth cohomology group is the space
$A^G$ of $G$-invariant elements in $A$. There exists a natural map

\[H^n_{lcm}(G,A)\to H^n_m(G,A).\] When $G$ and $A$ are
Lie groups and the $G$-action is smooth, the following
are natural maps between cohomology groups.
\[ H^n_{lsm}(G,A)\to H^n_{lcm}(G,A)\to H^n_m(G,A).\]

For any topological group $G$ and continuous $G$-module $A$, the 
first cohomology group 
\[H^1_{lcm}(G,A)=\lbrace c:G\ra A~|~c(st)=c(s)+s\cdot c(t),~c\textnormal{ is continuous}\rbrace.\]
When  $G,~A$ are locally compact, second countable topological groups, it follows
by a theorem of Banach, that measurable crossed homomophisms from $G$ 
to $A$ are continuous. Thus we have
\[H^1_m(G,A)=H^1_{lcm}(G,A)=H^1_{cont}(G,A).\]
For a Lie group $G$ and smooth $G$-module $A$ (a smooth $G$-module $A$ is an abelian Lie group such that 
the action $G\times A\to A$ is smooth) which is locally compact
and second countable,
the first cohomology group agrees with Moore cohomology group.
\[H^1_m(G,A)=H^1_{lcm}(G,A)=H^1_{lsm}(G,A).\]
which is the group of all smooth crossed homomorphisms from $G$ to $A$.
The advantage of working with locally regular cohomology can be 
seen in the holomorphic category.
\begin{prop}\label{firstholom}
Given a complex Lie group $G$ and a holomorphic $G$-module $A$,
 $H^1_{lhm}(G,A)$ is the group of all 
holomorphic crossed homomorphisms from $G$ to $A$.
\end{prop}

The main theorem of this paper is to 
show that  the second cohomology group $H^2_{lcm}(G,A)$
for locally compact second countable $G$ and $A$,
 parametrizes all the
locally split extensions of $G$ by $A$ .

\begin{theorem} 
\label{lcm}
If $G,~A$ are both locally compact, second countable topological groups, the
second cohomology group, $H^2_{lcm}(G,A)$ parametrizes
all the isomorphism classes of extensions  $E$ of $G$ by $A$,
\[1\ra A\overset\imath\ra E\overset\pi\ra G\ra 1\]
which are locally split.
\end{theorem}
It is this theorem that confirms our expectation that these locally regular cohomology theories can be a good and potentially useful cohomology theory for topological groups.  Our other attempts to constuct suitable cohomology theories failed to give a suitable interpretation for the second cohomology group. 

The proof of this theorem is a bit delicate.  Given a locally continous measurable $2$-cocycle, 
we construct an abstract extension group $E$. We topologize $E$ by first defining 
the product  topology in a sufficiently small `tubular' neighbourhood of identity in $E$, 
and by imposing the condition that left translations are continuous. To conclude that $E$ is a 
topological group, we need to show that inner conjugation by any element of $E$ 
is continuous at identity. For this,  we follow the
 idea of proof of Banach's theorem that a measurable homomorphism of locally
compact second countable groups is continuous. We find that the proof of Banach's theorem 
extends perfectly to prove that inner conjugations are continuous.  

In  the smooth category, we have
the following analogue of  Theorem \ref{lcm}. 
\begin{theorem}
Let  $G$ be a Lie group, $A$ be a smooth
$G$-module. The second cohomology groups, $H^2_{lsm}(G,A)$
parametrizes all the locally split smooth extensions
of $G$ by $A$.
\end{theorem}

Further, as a consequence of the positive solution to Hilbert's fifth problem, 
we have a comparison theorem as follows:
\begin{theorem}
Let  $G$ be a Lie group, $A$ be a smooth
$G$-module. Then the natural map, 
\[H^2_{lsm}(G,A)\to H^2_{lcm}(G,A)\]
is an isomorphism. 
\end{theorem}

The locally smooth measurable cohomology groups can be related to Lie algebra cohomology. 
This is easily done via the cohomology theory  $H_\square(G,A)$ 
based on germs of smooth cochains defined in a 
neighbourhood of identity developed by Swierczkowski (\cite[page 477]{swierczkowski}). 
We have a restriction map,
\[ H^n_{lsm}(G,A)\to H_\square^n(G,A), \]
given by restricting a locally smooth measurable cochain to a neighbourhood of 
identity in $G^n$ where it is smooth.

Now suppose $G$ acts on a finite dimensional real vector space $V$.
Let $L$ denotes the Lie algebra associated to the Lie group $G$ and let 
$H(L,V)$ denotes the Lie algebra cohomology.
It has been proved ({\it cf.} \cite[Theorem 2]{swierczkowski}) that 
\[H_\square(G,V)\simeq H(L,V).\]

\section{Continuous and measurable cohomology theories}
We briefly recall the earlier constructions of  $\check{\textnormal C}$ech
type cohomology theories of topological groups and their modules.  
Let $G$ be a topological group and $A$ be an abelian topological group. Assume that there is an action of $G$ on $A$. We say $A$ is a topological $G$-module if the group action,
\[ G\times A\to A,\]
 is continuous. 

Suppose $A', ~A, ~A''$ are topological $G$-modules and form a short exact sequence of abelian groups: 
\[0\rightarrow A'\overset\imath\rightarrow A\overset\jmath\rightarrow A''\rightarrow 0,\]
We say it is a short exact sequence of topological $G$-modules, if $\imath, ~\jmath$ are continuous, 
$\imath$ is closed, $\jmath$ is open  and there is an isomorphism of topological groups, 
\[ A/\imath(A')\simeq A''.\]

\subsection{Cohomology theory based on continuous cochains} 
\begin{defn}
\label{cochaincomplex}
Define the continuous cohomology $H^*_{cont}(G,A)$
as the cohomology  of the cochain complex $\lbrace C^n_{cont}(G,A),d^n\rbrace_{n\geq 0}$
given as:
\begin{enumerate}
\item
$C^0_{cont}(G,A)=A$
\item
For $n\geq 0$, $C^n_{cont}(G,A)$  denotes the group of all 
continuous maps from the product $G^n$ to $A$.
\end{enumerate}
The coboundary formula is given by the standard definition. For any $f\in C^n_{cont}(G,A)$, 
\begin{align*}
d^nf(s_1,s_2,\ldots,s_{n+1})=&
s_1\cdot f(s_2,\cdots,s_{n+1})+\underset{ i=1}{\overset{n}\sum}
(-1)^if(s_1,\ldots, s_is_{i+1},\ldots ,s_{n+1})\\
&+(-1)^{n+1}f(s_1,s_2,\ldots , s_n),
\end{align*}
\end{defn}

This theory coincides with the abstract cohomology theory, when the group is finite. 
When $G$ is profinite acting on a discrete abelian group, this cohomology theory is 
compatible with direct limits.  This theory was introduced 
by van Est ({\it cf.} \cite{van Est}).

\begin{example}\label{s1z}
Suppose  $G$ is a connected topological group and $A$ is  a discrete
$G$-module. Then it can be seen that  $G$ acts trivially on
$A$. The continuous cochains are just constant maps, and  
from the definition of the coboundary map, it can be seen that all the  
higher cohomology groups vanish. 

Consider the short exact sequence of trivial
  $G=S^1$-modules, 
\[ 0\ra \mathbb Z\ra \mathbb R\overset\pi\ra S^1\ra 1. \]

From the above remark, and the fact that there are no continuous
homomorphisms from $S^1$ to $ \mathbb R$, we see that,
\[H^1_{cont}(S^1,  \mathbb Z)=H^1_{cont}(S^1,  \mathbb R)=H^2_{cont}(S^1,  \mathbb Z)=0, \]
whereas 
\[H^1_{cont}(S^1, S^1)=\mbox{Hom}_{cont}(S^1, S^1)\simeq \mathbb Z.\] 
Hence this theory does not have long exact sequences of cohomology
groups even for the exponential sequence as above. 

Further, the second cohomology group
$H^2_{cont}(S^1,\mathbb Z)=(0)$. Thus, the second
continuous cohomology  does not account for even the natural exponential exact sequence 
given above. 
\end{example}

\subsection{Moore cohomology theory} 
The problem with continuous cohomology theory arises from the fact
that continuous cross sections do not exist in general. However,  we
have the following theorem guaranteeing the existence of measurable
sections: 
\begin{theorem}[Mackey, Dixmier]
Let $G$ be a polish group and $H$ be  a closed subgroup of $G$. Then
$H$ is a polish group and the projection map $G\to G/H$ 
 admits a measurable cross section.
\end{theorem}
We recall that a topological group $G$ is said to be polish if its
 topology admits a complete separable metric. 

Define  the group  $C^n_m(G,A)$ of measurable $n$-cochains with values in a
topological $G$-module $A$ to be the space  of all
measurable maps from $G^n$ to $A$. With  the coboundary map  as before, 
   we obtain a cochain complex $\lbrace C^n_m(G,A),d^n\rbrace$. 
We define the measurable or the Moore cohomology groups $H^n_m(G,A)$ are as 
cohomology groups $\lbrace C^n_m(G,A), d^n\rbrace$.

The measurable cohomology groups have many of the nice properties
required for a cohomology theory, viz., long exact
sequences of cohomology groups for any short exact sequence of
$G$-modules, correct interpretation of the low rank cohomology groups,
comparison with continuous cohomology, a form of Shapiro's lemma, weak
forms of Hochschild-Serre spectral sequences, etc. 

One of the difficult aspects of  Moore's theory, is that 
even if we restrict  ourselves to the category of topological
groups and continuous $G$-modules, to compute the cohomology, we
need  to work in the  measurable category. The measurable cochain
groups are large and it is difficult to compute them. 
For instance, the construction of the extension 
group corresponding to a measurable $2$-cocycle is
 not direct, and uses the topology on induced modules.

\subsection{Induced modules and extension groups}
Suppose  that $G, ~A$ are second countable,
 locally compact topological groups and $A$ is a topological $G$-module. 
Denote by $I(A)$ the group of measurable 
maps from $G$ to $A$. 

Since $A$ is Polish, there exists a metric  $\rho$ on $A$ whose
 underlying topology is the same as the original topology on $A$. 
Further, we can assume that $\rho$ is bounded. 
We take a finite measure $d\nu$ on $G$,
which is equivalent to the Haar measure on $G$ 
({\it cf.} \cite{moore3}). Define a metric on $I(A)$ as follows 
\[\bar\rho (f_1,f_2)=\underset G\int \rho(f_1(x),f_2(x))d\nu(x)\]
This makes $I(A)$  a Polish group. 
 We define a $G$-action on $I(A)$  by, 
\[(s\cdot f)(t)=sf(s^{-1}t)\quad \forall ~s,~t\in G, ~f\in U(G,A).\]
Via this action $I(A)$ becomes a topological $G$-module and  $A$ embeds as
 a $G$-submodule of $I(A)$ as submodule of the constant maps. 
It can be seen that the higher measurable cohomology groups 
of $I(A)$ are trivial ({\it cf.} \cite{moore3}).  We have a short exact sequence, 
\[ 0\to A\to  I(A)\to U(A)\to 0.\]

Moore showed that the second measurable 
cohomology group parametrizes the collection of extensions:
\begin{prop}
Suppose a Polish group $G$
acts continuously on an abelian Polish  group $A$.
Then $H^2_m(G, A)$ parametrizes the isomorphism classes of 
topological extensions of $G$ by $A$, 
\[1\rightarrow A\rightarrow E\rightarrow G\rightarrow 1.\]
\end{prop}
\begin{proof}
Given a topological extension
$E$ of $G$ by $A$, the construction of the $2$-cocycle
 corresponding to the extension is done (as before) using the existence of
 the measurable cross-section guarenteed by the theorem of Dixmier-Mackey.  

For our purpose, we briefly recall the proof of the converse. From the short exact sequence, and the
cohomological triviality of $I(A)$, we obtain an isomorphism, 
\[ H^2_m(G, A)\simeq H^1_m(G, U(A)).\]
Corresponding to a $2$-cocycle  $ b\in Z^2_m(G,A)$, we obtain  
a measurable crossed homomorphism, i.e., a continuous homomorphism 
 $T: ~G\to \colon U(A)\rtimes G$.  Let ${\mathcal E}$ be the image of $T(G)$. We have a short exact sequence, 
 \[ 0\to A\to  I(A)\rtimes G\to U(A)\rtimes G\to 0.\]
The required extension group $E$ is obtained as the inverse image of ${\mathcal E}$ in $I(A)\rtimes G$. 
We equip the group $E$  with  subspace topology of $I(A)\rtimes G$. It
can be verified  that $E$ is a closed subgroup of $I(A)\rtimes G$, 
and hence it is a Polish group. 
\end{proof}

\begin{remark}
From this construction, it does not seem possible to directly relate 
the topology of $E$ to that of $G$ and $A$; for example, if 
$G$ and $A$ are locally compact, will $E$ be
locally compact? This question was already answered
in the affirmative by Mackey ({\it cf.} \cite{mackey}), but the proof is neither 
easy nor direct. 

A similar problem arises when we work
with Lie groups and we want to relate the manifold structure on $E$ to
that of $G$ and $A$. This provides us another motivation (apart from the work of Lichtenbaum) 
 for  the construction of 
a cohomology theory based on measurable cochains which are
 continuous (or more generally are regular in a suitable sense) in a neighbourhood of identity. 
\end{remark}

\section{Locally continuous measurable cohomology theory}\label{chapter-false-tate}
We now describe our construction of a cohomology theory for topological groups, 
based on measurable cochains which are continuous in a 
neighbourhood of identity. We first establish  basic properties of this
 cohomology theory. Our main  aim is to show that  the
 second cohomology group describes the equivalence
 classes of topological extensions that are locally split. The
 construction of the extension allows us to directly relate the 
topology on the extension group to that of the subgroup and the quotient group.

\subsection{Definition and basic functorial properties}
We define the group $C^n_{lcm}(G,A)$ of locally continuous measurable 
$n$-cochains on $G$ with values in $A$ to be the space of measurable
 functions from $G^n$ to $A$  which are continuous in some neighbourhood 
of identity in $G^n$. The standard coboundary map involves the group action,
group multiplication on $G$, and composition in $A$. 
\[(s_1,s_2, \cdots s_{n+1})=s_1\cdot (s_2,s_3,\cdots  s_{n+1}),\]
\[(s_1,\cdots s_i,s_{i+1},\cdots s_{n+1})\mapsto (s_1,\cdots s_is_{i+1},\cdots s_{n+1}).\]
All these maps are continuous, and hence $d^n$ in Definition \ref{cochaincomplex}
 takes $C^n_{lcm}(G,A)$  to $C^{n+1}_{lcm}(G,A)$.
Therefore, $\lbrace C^\bullet_{lcm}(G,A), d^\bullet \rbrace_{n\geq 0}$ forms a
cochain complex.
For  every $ n\geq 0$, we denote the group of $n$-cochains by 
\[Z^n_{lcm}(G,A)=\lbrace f\colon G^n\ra A\colon\quad 
d^nf(s_1,s_2,\cdots ,s_{n+1})=0 \rbrace . \]
and we denote the group of $n$-coboundaries by
\[B^n_{lcm}(G,A)=\textnormal{Image }(d^{n-1})\textnormal{ for }n\geq 1.\]
We define $B^0_{lcm}(G,A)=0$. 
The locally continuous cohomology groups $H^n_{lcm}(G,A)$
for $n\geq 0 $ are defined as the  $n$-th cohomology
group of the cochain complex 
$\lbrace C^\bullet_{lcm}(G,A), d^\bullet\rbrace$, i.e., 
\[H^n_{lcm}(G,A)=\dfrac{Z^n_{lcm}(G,A)}{B^n_{lcm}(G,A)}.\]

It is clear that $H^0_{lcm}(G, A)=A^G$ the space of $G$-invariants in
$A$. These cohomology groups lie between the continuous cohomology and the
measurable cohomology, i.e., there are natural maps, 
\[ H^n_{cont}(G,A)\to H^n_{lcm}(G,A)\to H^n_{m}(G,A).\]
We will see that this cohomology theory addresses the problems related to
the cohomology theories discussed  earlier.
We first verify some of the basic properties satisfied by these
cohomology groups.
  
\subsection{Change of groups}
 Let $A,~A'$ be topological modules
for $G,G'$ respectively. Suppose there are continuous
homomorphism $\phi\colon G'\ra G$, $\psi\colon A\ra A'$ satisfying
 the following compatibility condition:
$$\begin{CD}
 G@.\times    @.   A   @>>> A\\
 @A{\phi}AA   @. @VV{\psi}V @VV{\psi}V\\
  {G'}@.\times@. {A'} @>>>{A'}
\end{CD}$$
\[g'\cdot\psi(a)=\psi(\phi(g')\cdot a)\quad\forall a\in A,~g'\in G'.\]
Then there is a map of cohomology groups, 
\[ H^n_{lcm}(G,A)\ra H^n_{lcm}(G',A'). \]
In particular, this gives functorial maps for  $G=G'$,
 \[ H^n_{lcm}(G,A)\ra H^n_{lcm}(G,A').\]
For $G'=H$, a subgroup of $G$, we have the
restriction homomorphism
\[ H^n_{lcm}(G,A)\ra H^n_{lcm}(H,A).\]

\subsection{Locally split short exact sequences}
\begin{defn} 
A short exact sequence of topological groups
 \[0\ra G'\overset\imath{\ra} G\overset\jmath{\ra} G''\ra 0\]
is an algebraically exact sequence of groups with the additional
property that $\imath$ is a closed and $\jmath$ is open. 
It is said to be  {\em locally split}, if the 
homomorphism $\jmath$ admits a local section, i.e., there 
exists an open neighbourhood $U''$ of identity in $G''$ and a
continuous map  $\sigma\colon U''\ra G$ such that
 $j\circ\sigma =id_{U''}$.
\end{defn}

We recall the definition of a finite dimensional topological space and topological group: 
\begin{defn}
 A topological space $X$ has
finite  topological dimension $k$,
 if every covering $\mathcal U$ of $X$ has a refinement $\mathcal U'$
 in which every point of $X$ occurs in at most $k+1$ sets 
in $\mathcal U'$, and $k$ is the smallest such integer. 
Finite dimensional topological groups
 are  topological groups that 
have finite dimension as a topological space.
\end{defn}

 The following theorem due to Mostert ({\it cf.} \cite[page 647]{mostert})
 provides examples of locally split short exact sequences: 
\begin{thmwn} 
Let $G$ be a finite dimensional locally compact group and
$H$ be a closed normal   subgroup of $G$, then $G/H$ admits  a local
cross-section.
\end{thmwn}
This theorem is obvious when $G$ is a Lie group, and it follows from
the fact that  that any finite dimensional locally compact group is 
an inverse limit of Lie groups. 
For detailed proof, refer to ({\it cf.} \cite{mostert}).
\begin{lemma}
\label{extendsection} 
Consider a locally split short exact sequence of topological groups 
 $1\ra G'\overset\imath{\ra} G\overset\jmath{\ra} G''\ra 1.$
   Then there exists a measurable section $\sigma: G''\to G$ which is
continuous in a neighbourhood of identity on $G''$. 
\end{lemma}
\begin{proof}
 By the hypothesis and Zorn's lemma, 
we can assume that there is a maximal
$Z\subset G''$  measurable set containing an open neighbourhood of
identity, and a section  $\sigma: Z\to G$
continuous in a neighbourhood of identity on $G''$. Suppose $Z$ is not
equal to $G$. By translation, given 
an element $w$ in the complement of $Z$ in $G''$ we can find a
neighbourhood $U_w$ of $w$ in $G''$,   and  a continuous section
$\sigma_w$ to $\jmath$
on $U_w$. Patching the sections $\sigma$ and $\sigma_w|_{U_w\cap
  (G''-\mathbb Z)}$ we get a measurable extension of $\sigma$. By maximality
this implies $Z=G''$. 
\end{proof}
\begin{remark}
\label{lindel}
When the group $G''$ is Lindel\"of (meaning every open cover 
of $G''$ admits a countable subcover), the above lemma can be 
proved easily without using Zorn's lemma. 
\end{remark}

\subsubsection{Long exact sequence}
As a corollary of Lemma
 \ref{extendsection},
 we associate to each \emph{locally split} short exact sequence,
a long exact sequence of locally continuous measurable cohomology
groups. 
\begin{cor} 
Consider  a  locally split short exact sequence of topological $G$-modules. 
$0\ra A'\ra A\ra A''\ra 0$. Then there is a short exact sequence of cochain complexes, 
\[ 0\ra
C^*_{lcm}(G,A')\overset{\tilde\imath}\ra C^*_{lcm}(G,A)\overset{\tilde\jmath}\ra C^*_{lcm}(G,A'')\ra 0.\]
Hence, there is a long exact sequence of locally continuous measurable
cohomology groups,
\[0\ra H^0_{lcm}(G,A')\ra  H^0_{lcm}(G,A)\ra
H^0_{lcm}(G,A'')\overset\delta{\ra} H^1_{lcm}(G,A')\ra\cdots \]

\end{cor}
\begin{proof}
 Since this is fundamental to our construction of cohomology theories,
 we briefly indicate the construction of the connecting homomorphism 
\[\delta: H^n_{lcm}(G, A'')\to H^{n+1}_{lcm}(G, A').\]
We choose a locally continous measurable  section $\sigma: A''\to A$ 
as given by Lemma  \ref{extendsection}. The connecting homomorphism is defined as, 
\[\delta(s)=d(\sigma\circ s), \quad s\in Z^n_{lcm}(G, A'').\]
This gives a well-defined cocycle with values in $A'$, and
 the cohomology class defined by this cocycle 
is independent of the choice of the section $\sigma$. 
\end{proof}

\subsubsection{The first cohomology group}
\begin{prop}
 Suppose $G$ is a topological group and and $A$ is a
topological $G$-module. Then
\[Z^1_{lcm}(G,A)=Z^1_{cont}(G,A)\]
is the space of continuous crossed homomorphisms from $G$ to $A$. 
\end{prop} 
\begin{proof} 
A locally continuous measurable $1$-cocycle is a measurable function
$c:G\to A$ satisfying  the cocycle condition
 \[c(s_1s_2)=s_1\cdot c(s_2)+c(s_1)\quad \forall s_1,s_2\in G.\]
Further, there exists an open set $U\subset G$ containing
identity such that $c|_U$ is   continuous. 
For any $x\in G$ arbitrary,
 the map $c|_{xU}$ satisfies  following formula.
\[c(xs)=x\cdot c(s)+c(x)\textnormal{ for all }s\in U.\]
Since the group action is continuous and
the map of  translation
by $c(x)$ is  continuous on $A$, we see that $c$ is 
continuous on $xU$.
\end{proof}

\begin{remark}
This holds in greater generality in the context of the measurable
cohomology groups constructed by Moore. Using Banach's theorem 
that any measurable homomorphism between two polish groups is
continuous, it can be seen that if $G$ and $A$ are locally compact and
$G$ acts continuously on $A$, then the first measurable cohomology
group is the group of all continuous crossed homomorphisms from $G$ to $A$.
\end{remark}

\subsection{Other constructions}
It is possible to construct other cohomology theories imposing different conditions 
on the nature of the cochains. For example, the proof of Lemma \ref{extendsection} 
can be modified to extend a
continuous section to a dense, open subset of
$G''$. We can then  construct a cohomology theory based on continuous
cochains defined on dense open subsets of products of $G$ (or even
measurable cochains which are continuous  on dense open subsets of
products of $G$).  However,
such a cohomology theory does not have restriction maps to subgroups
in general. Further, it seems difficult to relate the second
cohomology group (based on continuous
cochains defined on dense open subsets of products of $G$) to
extensions of $G$. 

Another construction can be based on set theoretic cochains which are
continuous in a neighbourhood of identity of $G^n$. But here again,
the second cohomology group does not seem to correspond to extensions of $G$
having local sections.

\section{Locally split extensions and $H^2_{lcm}(G,A)$} 
Let $G, ~A$ be locally compact second countable topological groups,
and assume that $A$ is a continuous $G$-module. 
Our aim out here is to establish a bijective
 correspondence between the second cohomology group $H^2_{lcm}(G,A)$ 
and the collection of all \emph{locally split} extensions $E$ of $G$
by $A$, i.e., those extensions for which there exists a continuous
section for the map $\pi\colon E\to G$ in some neighbourhood of identity.  

\begin{theorem}

 Let $G$ be a locally compact, second countable
topological group acting continuously on a locally compact, second
countable abelian group $A$. Then the second cohomology group
$H^2_{lcm}(G,A)$ parametrizes isomorphism classes of locally split  extensions of $G$
by $A$.
\end{theorem} 

Consider a locally split extension  $E$ of $G$ by $A$. 
We now associate a unique  
cohomology class in $H^2_{lcm}(G,A)$. By Lemma \ref{extendsection}, 
choose a measurable section  
$\sigma\colon G\ra E$ which is
continuous in a neighbourhood of identity in $G$. Define 
$f_{\sigma}\colon G\times G\ra A$ as
$f_{\sigma}(s_1,s_2)=\sigma(s_1)\sigma(s_2)\sigma(s_1s_2)^{-1}$. It
can be verified  that $f_{\sigma}$ satisfies the
cocycle condition and is continuous in a neighbourhood
of identity in $G\times G$. If we choose some other section $\tau$
having properties as above, then 
\[ f_{\tau}(s_1, s_2)=f_{\sigma} (s_1,s_2)+d(\sigma\tau^{-1})(s_1,s_2).\]
Since $\sigma,\tau:G\ra A$ are in $C^1_{lcm}(G,A)$, we see that
 $\sigma\tau^{-1}:G\ra A$ is measurable. It is continuous
around identity and hence it gives a $1$-cochain in $C^1_{lcm}(G,A)$. 
Therefore, $d(\sigma\tau^{-1})$ is a locally continuous $2$-coboundary 
and hence, both the sections  give  the same class in $H^2_{lcm}(G,A)$.

\subsection{Neighbourhoods of identity in $E$} 

Consider a measurable $2$-cocycle $F\colon G\times G\ra A$ which is continuous
on a neighbourhood $U_F\times U_F\subset G\times G$ of identity.
Since it is an abstract $2$-cocycle, we get an abstract extension
\[1\ra A\overset\imath\ra E\overset\pi\ra G\ra 1.\]

In order to topologize $E$, we define a base $\mathcal B$ for the neighbourhoods
 of identity in $E$ that consists of sets of the form 
$U_A\times U_G$, where $U_A$ and $U_G$ are open neighbourhoods of 
identity in $A$ and $G$ respectively, such that $F|_{U_G\times U_G}$
is continuous. It is easy to see that $\mathcal B$ is a filter base
in the terminology of Bourbaki ({\it cf.} \cite[Chapter 1, Section 6.3]{bourbaki}).
Let us call any subset of $E$, 
containing  some member of $\mathcal B$, as a neighbourhood of identity in $E$. 

Since the cocycle $F$ is continuous in  a neighbourhood of identity, it can be verified that the multiplication  map $E\times E\to E$ (resp. inverse map $E\to E$) are continuous in a neighbourhood of identity. From Proposition $1$ of Bourbaki
({\it cf.}\cite[Chapter 3, Section 1.2, page 221]{bourbaki}), 
for $E$ to be a topological group with $\mathcal B$ as a base for the 
neighbourhoods at identity, it is necessary and sufficient that 
inner conjugation by any element $a\in E$ is continuous at
  identity: for  $a\in\tilde E$ and any  
$V\in \tilde{\mathcal B}$, there exists $V'\in \tilde{\mathcal B}$
such that $V'\subset a\cdot V\cdot a^{-1}$.
 We single this out as a theorem: 
\begin{theorem} 
\label{maintheorem}
Let $E$ be an extension of the group $G$ by $A$ corresponding to 
the $2$-cocycle $F:G\times G\ra A$ and is provided with the neighbourhood
topology ${\mathcal B}$  defined above. Then for any  $x\in E$, the map of
inner conjugation $\imath_x:E\ra E$ is continuous at identity. 
\end{theorem}

Proof of this theorem will occupy next few sections. The proof 
of this theorem is modelled on the proof of Banach's theorem that
measurable homomorphims of second coutntable
 locally compact topological groups are
continuous. Heuristically, this can be considered as saying that 
the topology of a locally compact group can be
recovered from the underlying measure theory. 
Our proof of the above
theorem makes this heuristic precise. 

In our situation, $E$ can be
equipped with a measure structure since the cocycle $F$ is measurable.  
We topologize $E$ with a neighbourhood base filter ${\mathcal B}$,
imposing the condition that left translation is continuous. We show
that there exists a left invariant measure on $E$. This will allow us
to define convolution of measurable functions.  
We then use the fact that the multiplication and inverse maps
are continuous in a neighbourhood of identity $e$, together with a 
global argument to prove the above theorem.

\subsection{Topology on extension group}
We topologize $E$ by considering the left translates $x\mathcal B$ as 
a base for  the open neighbourhoods of $x\in E$. With this topology
left multiplication by any element $x\in E$ is a continuous map from
$E$ to $E$. It is easy to observe the following proposition 
listing some basic properties of the topological space $E$. 
\begin{prop}
\begin{enumerate}
 \item 
The homomorphisms $\imath,~ \pi$ are continuous and $\pi$ is an open map.
\item 
There exists an open neighbourhood $U_F\subset G$ of identity in
  $G$, and a continuous section $\sigma\colon U_F\ra E$.

\item
        The inclusion $\imath\colon A\ra E$ is 
	a homeomorphism onto its image 
	and $\imath(A)$ is a closed subset of $E$.
\item 
	 $E$ is a locally compact, second countable, Hausdorff space. 

\item The Borel algebra on $E$ is generated by members 
of the filters $\underset{ x\in E}\cup x\mathcal B$.
Moreover,  the measure structure on $E$ is  product of 
 measure structures on $G$ and $A$.
\item The group law
and the inverse map on $E$ are Borel measurable functions.
Hence, the map $\imath_x\colon E\ra E$ of inner conjugation 
by any $x\in E$ is Borel measurable.
\end{enumerate}
\end{prop}

\begin{proof}
Since $E$ is second countable, its Borel measurable sets are
generated by small open sets, namely the members of 
$\underset{x\in E}\cup x\mathcal B$.
We observe the formulae for group law and the inverse map:
\[(a_1,s_1)(a_2,s_2)=(a_1+s_1\cdot a_2+F(s_1,s_2),~s_1s_2),\]
\[(a,g)^{-1}=(s^{-1}\cdot(-a)+s^{-1}(-F(s,s^{-1})),~e_G).\]
The cocycle, $F:G\times G\ra A$ is measurable, and $G,~A$ are topological 
groups. Therefore, the group law
and inverse map are measurable on the product measure space 
$\mathcal M_A\times \mathcal  M_G$ which is the Borel measure
space $\mathcal M_E$ on $E$.
\end{proof}

\subsection{Construction of left invariant measure on $E$} 
In this section we construct a left invariant Borel measure on $E$. 
By Riesz representation theorem, it is equivalent to construct an
integral on $E$ which is invariant under left translation. For the 
construction of the left invariant integral, we follow the method of \cite[Chapter 3, Section 7]{fell}.
We remark that in our setting,  the only change 
is given by  Lemma \ref{uniformcontinuity}, analogous to the
uniform continuity lemma given by \cite[Chapter 2, Proposition 1.9]{fell}. 

Let $C_c(E)$ denote the space of 
real valued continuous functions with compact
support on $E$, and  $C_c^+(E)\subset C_c(E)$ the subspace of functions
taking nonnegative real values. We denote by $f,g,h$ the functions in
$C_c(E)$. For a function $f$ on $E$ and $u\in E$, let $f_u(x)=f(ux),~  x\in E$
denote the left translation of $f$ by $u$.  

\begin{lemma} \label{majorisation}
Let $E$ be as in
Proposition (\ref{maintheorem}), and let $f,~g\in C_c^+(E)$.
Let $g\neq 0$ with nonnegative values. Then there exist
finitely many positive real numbers
$c_1,c_2,\ldots ,c_r$ and elements $u_1,u_2,\ldots ,u_r\in E$ such
that
\begin{equation} 
f\leq c_1g_{u_1}+c_2g_{u_2}+\cdots +c_rg_{u_r}
\end{equation}
where $g_{u_i}\colon E\ra\mathbb R$
is defined as $g_{u_i}(s)=g(u_is)$, for all $s\in E$.
\end{lemma}
The proof follows from the compactness of the support of $f$.  This allows us to 
define the following: 

\begin{defn}
 Suppose $f, g\in C_c^+(E)$ are as above, 
we define the approximate
integral of $f$ relative to $g$ as
\[(f;g)=\mbox{inf}~\left\lbrace \underset{i=1}{\overset r\sum} c_i\right\rbrace ,\]
where the tuple $\left( c_1,c_2,\ldots ,c_r\right\rbrace $ runs 
over all the finite sequences of
non-negative numbers for which
 there exist group elements $u_1, u_2,\ldots ,u_r$ 
 satisfying the proposition above. By linearity, we 
define $(f;g)$ for any $f\in C_c(E)$.
\end{defn}

\begin{defn} 
Fix a compactly supported function $g:E\ra \mathbb R^+$.  
If $f,\phi\in C_c^+(E)$ and $\phi\neq 0$, define
$I_\phi(f)=(g;\phi)^{-1}(f;\phi)$
\end{defn}
\vspace{5mm}
It can be seen that the approximate integral $I_{\phi}(f)$ satisfies the 
following properties. The arguments are similar and follows from analogous properties satisfied 
by $(f;g)$ (see \cite[Chapter 3, Lemma 7.4 and page 202]{fell}):
\begin{lemma} 
\label{iphi}
Let  $f,f_1,f_2,\in C_c^+(E)$.  Then:
\begin{enumerate}
\item  If $f\neq 0$ then $(g;f)^{-1}\leq I_\phi (f)\leq (f;g)$;
\item $I_\phi (f_x)=I_\phi (f)~\textnormal{ for all } x\in G$;
\item $I_\phi (f_1+f_2)\leq I_\phi (f_1)+I_\phi (f_2)$;
\item $I_\phi (cf)=cI_{\phi} (f)$ for all $c\in\mathbb{R}_{\geq 0}$.
\end{enumerate}
\end{lemma}

 We next show that, if $\phi$ has small compact support,
$I_\phi$ is \textquotedblleft nearly additive\textquotedblright. For
this purpose, we require a lemma on uniform continuity, the analogue of 
\cite[Corollary 1.10, page 167]{fell}, whose proof we give since we do not yet have 
an uniform structure on $E$. 

\begin{lemma}
\label{uniformcontinuity} 
Let $f$ be a real valued  continuous function on $E$
and $\epsilon >0$. Suppose $C$ is a compact subset of $E$. 
Then, there is a neighbourhood $V$ of $e$ such that 
\[|f(x)-f(y)|<\epsilon\quad\textnormal{whenever }x, ~y\in C, ~x^{-1}y\in V.\]
\end{lemma} 
\begin{proof}
Suppose the lemma is not true. Then there exists $\epsilon>0$,
 a sequence $V_i$ of neighbourhoods of identity $e$  in $E$ 
with  $\cap V_i=\lbrace e\rbrace$,   elements
 $x_i,~y_i\in C$ with $x^{-1}_iy_i\in V_i$, such that 
\[|f(x_i)-f(y_i)|\geq\epsilon. \]
We can assume that $V_{i+1}V_{i+1}\subset V_i$.
Since $x_i,~y_i$ are in a compact set $C$, we can 
assume by passing to a subsequence, that the
sequence $x_i$ (resp. $y_i$) converges to $x_0$ (resp. $y_0$). Since
$f$ is continuous, 
\[|f(x_0)-f(y_0)|\geq\epsilon. \]

Since, $\lbrace x_i\rbrace$ converges to $x_0$, choose $N_k\geq k+1$
such that $x_i\in x_0V_{k+1}$ for $i\geq N_k\geq k+1$.
Now,  
\[y_i\in x_iV_i\subset (x_0V_{k+1})\cdot V_i.\]
\vspace{5mm}
Since  $i\geq k+1$,
\[\quad (x_0V_{k+1})\cdot V_i\subset(x_0V_{k+1})\cdot
V_{k+1}\subset x_0V_k.\]

Therefore, $y_i\in x_0V_k\quad\textnormal{whenever }  i\geq N_k$.
 Hence the sequence $\lbrace y_i\rbrace$ converges to $x_0$.
Since $f$ is continuous, this gives a contradiction. 

\end{proof}

We now prove that $I_{\phi}$ is nearly additive. 
\begin{lemma} 
Given $f_1,f_2\in C_c^+(E)$ and $\epsilon > 0$, we can
find a neighbourhood $V$ of $e$ such that, if $\phi\in C_c^+(E)$ is
non-zero and $supp(\phi)\subset V$, then
\begin{equation} 
|I_\phi(f_1)+I_\phi(f_2)-I_\phi (f_1+f_2)|\leq\epsilon
\end{equation}
\end{lemma}
The proof is similar to the proof of 
Lemma in  \cite[Chapter 3, 7.7]{fell}, only we use Lemma \ref{uniformcontinuity} in 
place of the lemma on uniform continuity available for a locally compact topological group. 
\begin{proof} 
Fix a non-zero function $f'\in C_c^+(E)$ which is
strictly positive on the $supp (f_1+f_2)$. We can find this function
because $E$ is locally compact and Hausdorff. Choose $\delta$ to be a small
positive number such that
	\begin{equation} 
	(f';g)\delta(1+2\delta)+2\delta(f_1+f_2;g) <\epsilon.
	\end{equation}
\vspace{5mm}
Now put $f=f_1+f_2+\delta f '$, and define
$$
h_i(x) = \begin{cases} \frac{f_i(x)}{f(x)} &\mbox{ if $
f(x) \neq 0 $}, \\ 
0 &\mbox{ otherwise}.
       \end{cases} 
$$
The functions $h_1,~h_2$ have compact support. 
By the left uniform continuity ({\it cf.} Lemma \ref{uniformcontinuity})
applied to $h_i, ~ i=1,2$, we can choose a neighbourhood $V$ of identity so that
\begin{equation} 
|h_i(x)-h_i(y)|<\delta ~\text{for} ~i=1,2\quad \textnormal{and }x^{-1}y\in V.
\end{equation}
We  can further assume by 
restricting to a smaller neighbourhood of identity $e\in E$ that  
multiplication and inverse maps are continuous on $V$.

Now let $\phi$ be any non-zero function in $C_c^+(E)$ with support
contained inside $V$.
We shall prove the lemma for $\phi$. By Lemma \ref{majorisation}, we can
find $0\leq c_j\in \mathbb
R$ and $u_j\in G$ such  that 
\[f\leq \underset{j=1}{\overset r\sum} c_j\phi_{u_j}.\]
 If $\phi(u_jx)\neq 0$, we have $u_jx\in V$. Therefore by Lemma \ref{uniformcontinuity}, 
\[|h_i(u_j^{-1})-h_i(x)|<\delta.\] 
Hence for $i=1,2$ and for every $x\in E$, we have
\[f_i(x)=h_i(x)f(x)\leq \underset{j=1}{\overset r\sum}
c_j\phi(u_jx)h_i(x)\leq \underset{j=1} {\overset
r\sum}c_j\phi(u_jx)(h_i(u_j^{-1})+\delta).\] 
This implies,
\[(f_i;\phi)\leq{\overset r{\underset{j=1}\sum}}c_j(h_i(u_j^{-1})+\delta), \quad i=1,2.\]
But $h_1+h_2\leq 1$ implies that
\[(f_1;\phi)+(f_2;\phi)\leq \underset{j=1}{\overset
r\sum}c_j(1+2\delta).\] 
Taking infimum over $\lbrace\sum c_j\rbrace$, we obtain
\[(f_1;\phi)+(f_2;\phi)\leq (1+2\delta)(f;\phi)\] 
Multiplying by $(g;\phi)^{-1}$, we get the
relation of relative integrals,
\[I_\phi(f_1)+I_\phi(f_2)\leq (1+2\delta)I_\phi(f).\]
By properties of $I_{\phi}$, and the choice  of $f'$, we see that
\[I_\phi(f_1)+I_\phi(f_2)\leq \delta I_\phi (f').\] 
Therefore we have,
\[ I_\phi(f_1)+I_\phi(f_2)\leq  I_\phi(f_1+f_2)+2\delta
I_\phi(f_1+f_2) +\delta (1+2\delta)I_\phi
(f')<I_\phi(f_1+f_2)+\epsilon.\] 
This proves the lemma.
\end{proof}

This completes the proof that $I_\phi$ is nearly additive, when $\phi$
has sufficiently small compact support. From Lemma \ref{iphi}
we also have 
\[(g;f)^{-1}\leq I_\phi(f)\leq (f;g)\]
whenever $g\neq 0$ and $f,\phi, g$ are compactly supported real valued
functions on $E$.

\begin{prop}
 There exists a non-zero left invariant integral on $E$.
\end{prop}

\begin{proof} 
For each $f\in C_c^+(E),~f\neq 0,$  let
$S_f=[(g;f)^{-1},(f;g)]$,  a  compact interval. Consider the
set $S=\underset{0\neq f\in C_c^+(E)}\prod S_f$ (with the cartesian
product topology). By Tychonoff's theorem $S$ is compact. Let $\lbrace
\phi_i\rbrace$ be a net of non-zero elements of $C_c^+(E)$ such that,
for each neighbourhood $V$ of identity, $supp(\phi_i)\subset V$ for
sufficiently large $i$.  By the properties of $I_{\phi_i}$, we know
that $I_{\phi_i}\in S$ for each $i$. Since $S$ is compact, we can
replace  $\lbrace\phi_i\rbrace$ by a subnet, and assume that
$I_{\phi_i}\ra I$ in $S$.  Putting $I(0)=0$, and passing to the
limit over $i$, we get  from  the properties of $I_{\phi_i}$ the following: 
 
\begin{enumerate}
 \item 
$(g;f)^{-1}\leq
I(f)\leq (f;g),\textnormal{ if }0\neq f\in C_c^+(E).$
\item
$ I(f_x)=I(f), \textnormal{ for all }f\in C_c^+(E);\textnormal{ for all }x\in E. $
\item
$ I(cf)=cI(f) \textnormal{for all } c\in\mathbb R_+;\textnormal{ and for all }f\in C_c^+(E). $
\item
$ I(f_1+f_2)=I(f_1)+I(f_2) \textnormal{ for all }f_1,f_2\in C_c^+(E).$
\end{enumerate}
 
Now any continuous function
$f:E\ra\mathbb C$ with compact support, can be written as 
$f=(f_1-f_2)+i(f_3-f_4)$, with each $f_i\in C_c^+(E)$.
Define
\begin{equation} 
I(f)=I(f_1)-I(f_2)+i(I(f_3)-I(f_4))
\end{equation}
\end{proof}
\begin{remark}
 This integral defines a left invariant measure on $E$,
Notice that the integral is positive, i.e., $I(f)>0$,
whenever $0\neq f\in C_c^+(E)$. Let $\mu$ be the Borel measure
on $E$, corresponding to the left invariant integral $I$ on $E$. If $x\in
G$, we see that \[\mu(W)=\mu(xW)\] where $W$ is a Borel subset of $E$
whose closure is compact. Further, if $K$ is any compact subset of $E$ then 
$\mu(K)<\infty$. 
\end{remark}

\subsection{Global argument}
We  now derive a consequence of the existence of a non-trivial left
invariant integral on $E$. We start with a general observation: 

\begin{lemma}
Let $G_1,~G_2$ be two groups, and  $f:G_1\ra G_2$ 
be a group homomorphism. Suppose that $G_1,~G_2$ 
are topological spaces, and $G_2$ satisfies Lindel\"of 
property. Assume further that there exist 
non-zero left invariant measures  $\mu_1$ (resp. $\mu_2$), on the 
Borel subalgebra of $G_1$ (resp. of  $G_2$ ). 

Let $f$ be measurable and $W\subset G_2$ be an arbitrary open subset.  Then 
\begin{itemize}
 \item 
The measure $\mu_2(W)> 0$. 
\item
If $f$ is surjective, the measure of the preimage $\mu_1(f^{-1}(W))> 0$.
\end{itemize}
\end{lemma}
\begin{proof}
Since $W$ is open in $G_2$ and $G_2$ is Lindel\"of, there exist
countably many left translates 
$\lbrace{t_iW\rbrace}_{i\in {\mathbb  N}}$ which cover $G_2$. 
Since the measure is left invariant and
non-zero, it follows that the measure of $W$ is positive. 

Since $f$ is surjective, there exist elements $s_i\in G_1$, such that
$f(s_i)=t_i$. Since $f$ is measurable, the inverse image $f^{-1}(W)$
is a measurable subset of $G_1$.
 Since $\lbrace{t_iW\rbrace}_{i\in {\mathbb   N}}$ cover $G_2$, 
the collection $\lbrace{s_if^{-1}(W)\rbrace}_{i\in {\mathbb
    N}}$ covers $G_1$. Since the measure $\mu_1$ is non-zero on $G_1$, it follows that 
$\mu_1(f^{-1}(W))> 0$.

\end{proof}

We apply this global argument when  $E=G_1=G_2$ with $f=\imath_x$
inner conjugation by an element $x\in E$:

\begin{cor} \label{globalargument}
Let $E$ be as above and $\mu$ denote the left invariant measure constructed in the foregoing subsection. Let $W$ be an open subset of $E$ and $x$ an element of $E$.  Then 
\[ \mu(i_x^{-1}(W))>0.\]
\end{cor} 

\subsection{Convolution}
The proof of Banach's theorem for locally compact groups 
proceeds by first showing that
 convolution of measurable functions satisfying suitable properties is continuous.
In our context, we can carry  out such an argument for measurable functions supported in a 
sufficiently small neighbourhood of identity in $E$. However, here we establish directly 
a statement that suffices for proving Theorem \ref{maintheorem}. The proof makes more use of symmetric subsets, has the advantage of simplifying the required arguments in our context by reducing the requirement of uniform continuity to Lemma \ref{uniformcontinuity}. 
The key proposition is the following:
\begin{prop}
 \label{keyproposition}
Let $M$ be a measurable, symmetric (i.e. $M=M^{-1}$) subset of $E$.
Suppose that $M\subset \pi^{-1} (U_2)$, for $U_2$ a symmetric relatively compact 
open neighbourhood of identity in $G$ such that the product of the closures 
 $\overline{U}_2\overline{U}_2\subset U_F$.  Assume that identity $e\in M$
and measure $\mu (M)$ is positive and finite.
 Then the set \[ MM=\lbrace xy\colon x\in M, y\in M\rbrace\]  
contains an open neighbourhood of identity in $E$. 
\end{prop}
\vspace{5mm}
Granting this proposition, we now prove Theorem \ref{maintheorem}.
\begin{proof}[Proof of Theorem \ref{maintheorem}]
 We need to show
that for any sufficiently small neighbourhood $V$ of identity in $E$
the set $\imath_x^{-1}(V)$ contains an open neighbourhood of identity $e\in E$. 
Let $W$ be a symmetric open neighbourhood of $e$ in $E$ satisfying  following
\begin{enumerate}
\item
$W$ is symmetric (i.e., $W=W^{-1}$)
\item
$W\subset \pi^{-1}(xU_2x^{-1})$
\item
$WW\subset V$

\end{enumerate}
Let $M'=\imath_x^{-1}(W)$. By Corollary \ref{globalargument},  $\mu (M')>0$.
Since  $W\subset \pi^{-1}(xU_2x^{-1})$, we have $\imath_x^{-1}(W)\subset \pi^{-1}(U_2)$.
Intersecting wih a symmetric compact set $K$ containing identity $e\in E$, we can assume
that $M=\imath_x^{-1}(W)\cap K$ has finite, positive measure, and is contained inside  $\pi^{-1}(U_2)$. By Proposition \ref{keyproposition},
we see that the product set $MM$  contains an open neighbourhood
$V'_x$ of $e$. Now
\[\imath_x^{-1}(V)\supseteq \imath_x^{-1}(WW)\supset MM\supset V'_x. \]
This proves Theorem \ref{maintheorem}.  
\end{proof}
\vspace{5mm}
We now proceed to the proof of Proposition \ref{keyproposition}.
For $x\in E$, define the function 
\[u(x)=\mu(M\cap xM).\]
The proof of the proposition  reduces to the following lemma:
\begin{lemma}
\label{keylemma}
 Under the hypothesis of Proposition \ref{keyproposition}, $u$ is  a 
continuous function. 
\end{lemma}
\vspace{5mm}
Assuming Lemma \ref{keylemma} we now prove Proposition \ref{keyproposition}.
\begin{proof}[Proof of Proposition \ref{keyproposition}]
 If $u(x)\neq 0$. them $M\cap xM\neq \emptyset$. Hence $x\in MM^{-1}=MM$
as $M$ is assumed to be symmetric. Further, $u(e)=\mu(M)>0$. Since $u$ is
continuous, this proves  Proposition \ref{keyproposition}.
\end{proof}
\vspace{5mm}
We now proceed to the proof of Lemma \ref{keylemma}. 
\begin{proof}[Proof of Lemma \ref{keylemma}.]
 Let $\chi _M$ denote the characteristic function of $M$. Then $u(x)$ can
be defined by the following integral.
\[
\begin{split}
u(x)&=\mu(M\cap xM)\\
&=\underset{M\cap xM}{\int}d\mu(y)\\
&=\underset M{\int}\chi_M(x^{-1}y)d\mu(y)\\
&=\underset E{\int}\chi_M(y)\chi_M(x^{-1}y)d\mu(y).
\end{split} \]
We observe that support of $u$ is contained
 inside $MM\subseteq U_2U_2$. Since $M\subset \pi^{-1}(U_2)$, 
by Lusin's theorem choose 
a function $f\in C_C(E)$ with support contained inside  $\pi^{-1}(\overline{U}_2)$ such that
\[\underset E{\int}|\chi_M(y)-f(y)|~d\mu(y)<\epsilon_1, \]
for  some sufficiently small $\epsilon_1>0$.

Let $\lbrace x_n\rbrace_{n\in\mathbb N}$ be a sequence converging to
 $x_0$ in $\pi^{-1}(U_2U_2)$.
To show the continuity of $u$ restricted to $\pi^{-1}(U_2U_2)$, it is enough to show that
the sequence $\lbrace u(x_n)\rbrace$ converges to $u(x_0)$. We have 
\[|u(x_n)-u(x_0)|=\left|\underset E{\int}\chi_M(y)(\chi_M(x_n^{-1}y)-\chi_M(x_0^{-1}y))d\mu(y)\right|.\]
Since $M$ is symmetric, we also have  $\chi_M(y^{-1})=\chi_M(y)$. 
\vspace{5mm}
Therefore, 
\[|u(x_n)-u(x_0)|\leq \underset E{\int}\chi_M(y)|\chi_M(y^{-1}x_n)-\chi_M(y^{-1}x_0)|d\mu(y).\]
We have 
\[
\begin{split}
|u(x_n)-u(x_0)|& \leq \underset E{\int}\chi_M(y)|\chi_M(y^{-1}x_n)-\tilde{f}(y^{-1}x_n)|d\mu(y)\\
& +\underset E{\int}\chi_M(y)|\tilde{f}(y^{-1}x_n)-\tilde{f}(y^{-1}x_0)|d\mu(y)\\
&+  \underset E{\int} \chi_M(y)|\tilde{f}(y^{-1}x_0)-\chi_M(y^{-1}x_0)|d\mu(y), 
\end{split}
\]
where $\tilde{f}(z)=f(z^{-1})$ for $z\in E$. 
Since the integral is left invariant, by replacing $y$ by $x_ny$ (resp. by $x_0y$) 
in the first (resp. third) term on the right, we see that 
\[\begin{split}
\underset E{\int}\chi_M(y)|\chi_M(y^{-1}x_n)-\tilde{f}(y^{-1}x_n)|d\mu(y)&= \underset E{\int}\chi_M(x_ny)|\chi_M(y^{-1})-\tilde{f}(y^{-1})|d\mu(y)  \\
  &\leq \underset E{\int}|\chi_M(y^{-1})-\tilde{f}(y^{-1})|d\mu(y)\\
&= \underset E{\int}|\chi_M(y)-f(y)|d\mu(y)\\
& < \epsilon_1.
\end{split}
\]
Here we have used the fact that $M$ is symmetric and 
 definition of $\tilde{f}$.  Similarly, we obtain
\[\underset E{\int} \chi_M(y)|\tilde{f}(y^{-1}x_0)-\chi_M(y^{-1}x_0)|d\mu(y)<\epsilon_1.\]
Now we estimate the middle term.  
Since  inverse map is continuous in $\pi^{-1}(U_F)$ and 
 support of $f$ is contained inside  $\pi^{-1}(U_F)$, 
the function $\tilde{f}(y)=f(y^{-1})$ is continuous.
 Given $\epsilon_2>0$, by Lemma \ref{uniformcontinuity},
 there exists a symmetric neighbourhood $W_2$ contained
 inside $\pi^{-1}(U_F)$ (here again we are using  fact that
 inverse map is continuous on $\pi^{-1}(U_F)$) such that 
\[ |\tilde{f}(z_1)-\tilde{f}(z_2)|<\epsilon_2\quad \mbox{for} \quad z_1^{-1}z_2\in W_2.\]
Since $x_n$ converges to $x_0$, there exists 
a natural number $N$ such that for $n\geq N$, 
\[x_n\in x_0W_2. \quad i.e., \quad x_0^{-1}x_n\in W_2.\]
Since $W_2$ is symmetric, this condition can be rewritten as 
\[ x_n^{-1}x_0\in W_2.\]
Hence, 
\[ (y^{-1}x_n)^{-1}(y^{-1}x_0)= x_n^{-1}x_0\in W_2.\]
By applying  Lemma \ref{uniformcontinuity}  to  the continuous 
function $\tilde{f}$, we obtain
\[|\tilde{f}(y^{-1}x_n)-\tilde{f}(y^{-1}x_0)|\leq \epsilon_2,
\textnormal{ for } n\geq N.\]
 Hence for $n\geq N$, the middle term can be estimated as
\[ \underset E{\int}\chi_M(y)|\tilde{f}(y^{-1}x_n)-\tilde{f}(y^{-1}x_0)|d\mu(y)<
\epsilon_2\underset E{\int}\chi_M(y)<\epsilon_2\mu(M).
\]
Combining the above estimates, we obtain  
\[|u(x_n)-u(x_0)|<2\epsilon_1+\epsilon_2\mu(M)~\textnormal{ for all }n\geq N.\]
This establishes continuity of $u$ and hence proves Lemma \ref{keylemma}.
\end{proof}
\subsection{Comparison with other cohomology theories}
Suppose $G$ is a locally compact group acting on 
a locally compact group $A$. 
Then we have a natural map
\[ H^2_{lcm}(G,A)\rightarrow H^2_m(G,A).\]
As a corollary to Banach's theorem we show that the above map is injective: 
\begin{cor}
Let  $G, ~A$ be   locally compact, second countable groups.  Then the natural map
\[ H^2_{lcm}(G,A)\rightarrow H^2_m(G,A)\]
is injective.  
\end{cor}
\begin{proof}
 Suppose a $2$-cohomology class $\underbar c$ in $H^2_{lcm}(G,A)$
is trivial in  $H^2_m(G,A)$. 
Construct the corresponding extension $E$
of $\underbar c$. Then $\underbar c=0$ in  $H^2_m(G,A)$, implies that 
there exists a measurable section $\sigma:G\rightarrow E$ which is a 
group homomorphism. By Theorem \ref{maintheorem}, we know that $E$ is locally compact. It can be seen that $E$ is also second countable. Hence by 
 Banach's theorem,   $\sigma$
is a continuous group homorphism, and this implies that the extension $E\simeq A\rtimes G$. 
\end{proof}

\begin{cor}
 Suppose that either of the following conditions hold:
\begin{enumerate}
 \item $G$ is a profinite group and $A$ is a discrete $G$-module. 

\item $G$ is a Lie group and $A$ is a finite dimensional vector space. 
\end{enumerate}
Then the natural maps, 
\[H^2_{cont}(G,A)\rightarrow  H^2_{lcm}(G,A)\rightarrow H^2_m(G,A),\]
are isomorphisms. 
\end{cor}
\begin{proof} This follows from the previous corollary and the isomorphism 
 \[ H^2_{cont}(G,A)\rightarrow H^2_m(G,A)\]
for the given cases ({\it cf.} \cite{moore3}). 
\end{proof}

\section{Cohomology theory for Lie groups} 
In this chapter, we work in the smooth category in
the context of Lie groups $G,~A$ with smooth actions $G\times A\rightarrow A$.
Here we can define an analogous cohomology theory 
$H^n_{lsm}(G,A)$ where we impose 
the condition that the measurable cochains are locally smooth and
study some of its basic properties. We show in the context of Lie groups 
that the second locally smooth measurable cohomology group $H^2_{lsm}(G,A)$
parametrizes the collection of locally split extensions of $G$ by $A$. Further  
we observe as a corollary to the solution of Hilbert's Fifth problem and 
Theorem \ref{maintheorem}, an isomorphism of $H^2_{lcm}(G,A)$ with 
$H^2_{lsm}(G,A)$. 

Let $G$ be a Lie group and $A$ be a smooth $G$-module.
Analogous to  the construction in the continuous case,  we form a cochain
complex  $\lbrace C^n_{lsm}(G,A),d^n\rbrace_{n\geq 0}$. This starts with
$C^0_{lsm}(G,A)=A$, and for higher  $n, ~C^n_{lsm}(G,A)$ is the group
of all measurable functions $f\colon G^n\ra A$, which are smooth
around identity. It is easily checked that the standard coboundary operator
restricts to define a cochain complex. 

\begin{defn} 
The locally smooth cohomology theory 
 $\lbrace H^n_{lsm}(G,A)\rbrace_{n\geq 0}$ is defined as the  
cohomology groups of the cochain complex 
$\lbrace  C^n_{lsm}(G,A);d^n\rbrace_{n\geq 0}$.
\end{defn}
It is clear that there exists a map $ H^n_{lsm}(G,A'')\rightarrow H^n_{lcm}(G,A'')$.
It is easy to establish the following properties of  $H^n_{lsm}(G,A)$.
\begin{prop}
 \begin{enumerate}
  \item
   $H^0_{lsm}(G,A)=A^G.$
  \item
  The first cohomology group $H^1_{lsm}(G,A)$ is the group 
of all smooth crossed homomorphisms from $G$ to $A$.
  \item
     Given a short exact sequence of $G$-modules 
\[0\rightarrow A\overset\imath{\rightarrow }A\overset\jmath{\rightarrow}A\rightarrow 0,\]
there is a long exact sequence of cohomology groups
\[\cdots\rightarrow H^i_{lsm}(G,A')\rightarrow H^i_{lsm}(G,A)\rightarrow H^i_{lsm}(G,A'')\overset\delta
{\rightarrow} H^{i+1}_{lsm}(G,A)\rightarrow\cdots.\]
 \end{enumerate}
\end{prop}
\begin{proof}
The proof of (ii) follows from the smoothness at identity and the cocycle condition. 
For (iii), given a short exact sequence of $G$-modules, 
it follows from the property that the sequence is locally split
(i.e.  $\jmath$ admits a smooth local section). 
We extend the section to a locally smooth measurable
section $\sigma:A''\rightarrow A$. Arguing as before, we obtain the long 
exact sequence of cohomology groups. 
\end{proof}
\begin{remark}
We can introduce an analogous cohomology theory in the holomorphic context based on 
measurable cochains which are holomorphic in a neighbourhood of identity. In this context, 
we observe that the first cohomology group 
$H^1_{lhm}(G,A)$ is the space of all holomorphic crossed homomorphisms from $G$ to $A$ (see Proposition 
\ref{firstholom}). Since the closed graph theorem is not applicable in the holomorphic context, we cannot obtain 
this result from the smooth version by an application of  arguments as above. 
\end{remark}

\subsection{Extensions of Lie groups}
We describe now the  second cohomology group.
\begin{theorem}
 \label{lsm-Lie}
Let $G$ be a Lie group and $A$ be a smooth $G$-module. Then 
the  second cohomology group $H^2_{lsm}(G,A)$  parametrizes equivalence
classes of extensions  $E$ of $G$ by $A$,
 \[1\rightarrow A\overset\imath{\rightarrow} E\overset\pi{\rightarrow}G\rightarrow 0,\]
where $E$ is a Lie group with a measurable cross section $\sigma:G\rightarrow E$ such
that $\sigma$ is smooth around identity in $G$. 
\end{theorem}
\begin{proof}
Given an extension $E$ of $G$ by $A$ with  a locally smooth measurable cross section
 $\sigma:G\rightarrow E$, we assign to it the  $2$-cohomology
 class of $F_\sigma :G\times G\rightarrow A\in Z^2_{lcm}(G,A)$ 
that takes $(s_1,s_2)\in G\times G$
to $F_\sigma(s_1)F_\sigma(s_2)F_\sigma(s_1s_2)^{-1}$.

For proving  converse, 
we take an arbitrary cohomology class $\bar F\in H^2_{lsm}(G,A)$. Choose
a representative $F$ and construct an abtract group extension $E$ 
\[1\rightarrow A\overset\imath\rightarrow  E\overset\pi\rightarrow G\rightarrow 1 .\]
Suppose $G^0$ is the connected component of identity in $G$, 
we claim that the  subgroup $E^0:=\pi^{-1}(G^0)\prec E$  is a Lie group.
Since $G^0$ is a normal subgroup of $G$ and $\pi$ is surjective, 
the subgroup $E^0$ is normal in $E$.
We shall first verify that  $E^0$  is a Lie group
(we remark here that we can carry out a similar argument in the 
continuous case to show directly that $E^0$ is a topological group, 
instead of using  Theorem \ref{maintheorem}). We then use the
 measurability condition to show that
 the extension $E$ is a  Lie group.

Since the cocycle $F$ is smooth in a neighbourhood of identity, 
we assume that $U_G$ is sufficiently small  so that the following holds: 
\begin{itemize}
 \item The product map
\[ (x, y, z)\mapsto xyz \]
is smooth from $U_G\times U_G\times U_G$ to $U_F$. This can be ensured by assuming that 
the  following functions are smooth on $U_G\times U_G\times U_G$:
\begin{equation}
 \label{multiplyU_G}
(s_1,s_2,s_3)\mapsto F(s_1s_2, s_3) \quad \mbox{and}\quad (s_1,s_2,s_3)\mapsto F(s_1,s_2s_3).
\end{equation}
\item 
The map $s\mapsto s^{-1}$ is smooth from $\pi^{-1}(U_G)$ to $\pi^{-1}(U_G)$ (here we have assumed
$U_G$ is symmetric). 
\end{itemize}
We define an atlas on $E^0$ by imposing that left translations 
are diffeomorphisms and imposing the 
product of smooth structure on 
$\pi^{-1}(U_G)\simeq A\times U_G$, i.e., the atlas consists of 
 $(xU, \phi\circ L_{x^{-1}})$, where $x\in E^0$ and $U$ is 
an open subset of $\pi^{-1}(U_G)$.
Here $(U, \phi) $ is a part of the atlas for the product
 smooth structure on  $\pi^{-1}(U_G)$. 

We first claim that this gives us an atlas:
 suppose there exists elements $x, ~y\in E^0$ and open sets 
$U, ~V$ contained inside $\pi^{-1}(U_G)$
 such that $xU\cap yV\neq \emptyset$.
 By taking the union of $U$ and $V$,
 we can assume that $U=V$. We have the charts,
\[ xU \xrightarrow{L_{x^{-1}}}U \xrightarrow{\phi} W\]
\[ yU \xrightarrow{L_{y^{-1}}}U \xrightarrow{\phi} W\]
where $W$ is an open subset in some Euclidean space.
 Let $V=U\cap x^{-1}yU$ be the image of $L_{x^{-1}}(xU\cap yU)$. 
We need to show that the map,
\[ \phi\circ L_{y^{-1}}\circ L_x\circ \phi^{-1}: \phi(V)\to W\textnormal{ is smooth.}\]
For this it is enough to show that 
\[ L_{y^{-1}x}: V\to U \textnormal{ is smooth.}\] 
The hypothesis implies that there exists elements
 $z, ~z'$ in $U$ such that $xz=yz'$, i.e., $y^{-1}x=z'z^{-1}$. 
This implies that 
\[y^{-1}x \in \pi^{-1}(U_G)\times \pi^{-1}(U_G).\]
Hence the required smoothness follows from the
 assumption that the triple product is smooth from
 $\pi^{-1}(U_G)\times \pi^{-1}(U_G)\times \pi^{-1}(U_G)$ 
to $\pi^{-1}(U_F)$. 
This concludes the proof that $E^0$ with the above atlas
 is a smooth manifold. We remark that the manifold structure
 is such that left translations are diffeomorphisms.  

We now have to show that $E^0$ is a Lie group.
 For this we first observe that inner conjugation
 by any element $x\in E^0$ is smooth at identity.
 Since $G^0$ is a connected Lie group, the neighbourhood $U_G$ 
generates $G$ as a group. It follows that 
the  group $E^0$ is generated by $\pi^{-1}(U_G)$. 
Hence any element $x\in E^0$ can be written as 
\[ x=x_1\cdots x_r,\textnormal{ where each }x_i\in \pi^{-1}(U_G).\] 
By our choice of $U_G$, inner conjugation by any
 $x_i\in \pi^{-1}(U_G)$ is smooth at identity. Since 
the inner conjugation by $x$ is a composite of
 inner conjugations by the elements $x_i$,  it follows that 
inner conjugation by any element of $x\in E^0$ is smooth at identity. 

We now show that the multiplication map $E^0\times E^0\to E^0$ is smooth. Suppose $x, ~y\in E^0$. Let $U$ be a sufficiently small neighbourhood of identity in $E$ such that the conjugation map $z\mapsto y^{-1}zy$ is smooth where $z\in U$. Now the multiplication map $xU\times yU$ can be written as, 
\[ (xz)(yz')=(xy)(y^{-1}zy)z' \quad z, ~z'\in U.\]
We can assume that $U, ~y^{-1}Uy\subset \pi^{-1}(U_G)$. Since left multiplication by $xy$ is smooth, and multiplication is smooth on  $\pi^{-1}(U_G)\times \pi^{-1}(U_G)$, we conclude that multiplication is a smooth
map from $E^0\times E^0$ to $E^0$. 

Similarly, to show that the inverse map is smooth
 on $E^0$, say around $x\in E^0$, we take  $U$ to be a  
sufficiently small neighbourhood of identity in $E^0$ such that
 $z\mapsto x z^{-1}x^{-1}, ~z\in U$ is smooth on $U$. 
(we use the fact that inverse map is smooth 
on $\pi^{-1}(U_G)$ and assume that $U\subset \pi^{-1}(U_G)$). Now,
\[ (xz)^{-1}=x^{-1}(xz^{-1}x^{-1}), \quad z\in U.\]
As left translations are smooth, it follows that the 
inverse map is smooth on $E^0$. This concludes 
the proof that $E^0$ is a Lie group. 

\begin{remark}
We remark again out here, that the above arguments did not require to 
start with that $E$ or $E^0$ is a topological group.
 The above arguments, carried out in the continuous category,
 will  directly yield that $E^0$ is a topological group. We have only
 used the fact that any neighbourhood of identity in $G$ generates $G$ as 
a group and that cocyles are locally regular (locally regular means locally continuous
or locally smooth depending on the setting).  
\end{remark}
Now we want to conclude that $E$ is a Lie group. For 
this, we first show that $E$ is a topological group.
Since the cocycle is measurable, we see that inner conjugation 
$i_x$ by any element $x\in E$ is a measurable automorphism of $E^0$. 
By Banach's theorem, it follows that $i_x$ is continuous on $E^0$. 
(in particular it follows that $E$ is a topological group). 

Since $\imath_x$ is continuous on $E^0$, the graph of $\imath_x$ is 
closed in $E^0\times E^0$. 
Therefore, the graph of $\imath_x$
is a closed subgroup of the Lie group  $E^0\times E^0$.
Therefore,the  graph of $\imath_x$ is a Lie group of $E^0\times E^0$. 
Therefore, that $i_x$ is a smooth diffeomorphism of $E^0$. 
We now argue as above to conclude that $E$ is a Lie group. 
Therefore we get the following short exact sequences of topological groups,
\[1\rightarrow A\overset\imath\rightarrow E\overset\pi\rightarrow G\rightarrow 1.\]

Since $E$ and $G$ are Lie groups with a continuous 
group homomorphism $\pi: E\rightarrow G$,
we see that graph of $\pi$ is a closed subgroup of 
$E\times G$ which is a Lie group. 
Therefore graph of $\pi$ is a Lie subgroup. 
This implies that $\pi$ is smooth. By implicit function
theorem, $\pi$ admits a  smooth cross section in a neighbourhood of identity. 
We use  arguments similar to those
 used in proving Lemma \ref{extendsection}
and extend this to a locally smooth measurable
 cross section $\sigma$ from $G$ to $E$.
This concludes proof of Theorem \ref{lsm-Lie}. 
\end{proof}
\subsection{A comparison theorem}
In this section, as a corollary of 
 positive solution to Hilbert's fifth problem, we
show the following: 
\begin{theorem}
Let $G$ be a Lie group and $A$ be a smooth $G$-module. Then the natural map, 
\[H^2_{lsm}(G,A)\to H^2_{lcm}(G,A),\]
is an isomorphism. 
\end{theorem}
\begin{proof}
Let  $F:G\times G\rightarrow A$ be a locally continuous
 measurable $2$-cocycle on $G$ with values in $A$.
 We shall show that $F$ is cohomologous to a locally smooth measurable 
$2$-cocycle $b$.
\vspace{5mm}
By Theorem \ref{maintheorem}, we obtain a locally split (topological) extension
$E$ of $G$ by $A$:
\[1\rightarrow A\overset\imath\rightarrow E\overset\pi\rightarrow G\rightarrow 1.\]
Denote by $E^c$ the connected component of $E$ containing  identity.
Since the extension is locally split and $G$ and $A$ are Lie groups,
 it follows that  $E^c$ is locally Euclidean. Hence  by 
 positive solution to Hilbert fifth problem 
({\it cf.} \cite{yamabe},\cite{gleason},\cite{zippin}),
 we conclude  that $E^c$ is a Lie group. 

Now the map $\pi|_{E^c}\rightarrow G$ is a  continuous homomorphism.
 Hence the graph of $\pi|_{E^c}$
is a closed subgroup of the Lie group $E^c\times G$. 
Therefore,  it is a Lie subgroup and this shows that the
 projection map   $\pi|_{E^c}$ is smooth. Applying the 
 implicit function theorem, we can find a smooth cross section
of $\pi$ in a neighbourhood of identity on $G$ to $E^c$.
 By arguments similar
to Lemma \ref{extendsection}, we extend this to a measurable section $\sigma$
from $G$ to $E$.

The section  $\sigma$ gives raise to  a $2$-cocycle $b_\sigma:G\times G\rightarrow A$ in $Z^2_{lsm}(G,A)$
given by the formula $b_\sigma(s_1,s_2)=\sigma(s_1)\sigma(s_2)\sigma(s_1s_2)^{-1}$.
One can observe that $b_\sigma$ is cohomologous to $F$ in 
$Z^2_{lcm}(G,A)$. This yields a surjective map 
\[H^2_{lsm}(G,A)\rightarrow H^2_{lcm}(G,A).\]

We, next claim this map to be injective.
Suppose a class $\underbar b\in H^2_{lsm}(G,A)$
 is trivial in $H^2_{lcm}(G,A)$.
Corresponding to $\underbar b\in H^2_{lsm}(G,A)$,
 by Theorem \ref{lsm-Lie} we obtain a Lie group
$E$ which is an extension of $G$ by $A$ 
Since  $\underbar b=0$ in $H^2_{lcm}(G,A)$, 
there exists a locally continuous measurable section
$\sigma:G\rightarrow E$ which is a group homomorphism.
 Since it is continuous at identity, it is 
continuous everywhere. Hence we obtain a 
continuous isomorphism between the Lie groups 
$E$ and $A\rtimes G$. By an application of the 
closed graph theorem, this isomorphism is smooth. 
Therefore, the cohomology class $\underbar b$ is
 trivial in $H^2_{lsm}(G,A)$.
\vspace{5mm}
Hence it follows that 
 \[H^2_{lsm}(G,A)\rightarrow H^2_{lcm}(G,A)\]
is an isomorphism.
\end{proof}

\end{document}